\documentclass[12pt]{amsart}
\usepackage{amscd,amsthm,amsfonts,amssymb,amsmath,euscript}
\usepackage{amsmath}
\usepackage{amsfonts}
\usepackage{amssymb}

\newtheorem{theorem}{Theorem}[section]
\newtheorem{lemma}{Lemma}[section]
\newtheorem{corollary}{Corrolary}[section]

\begin{document}

\centerline {\bf Klein Foams }

{\ }

\centerline {\bf Antonio F. Costa {\footnote {Partially supported
by MTM2008-00250}},}

\centerline{Departamento de Matemсticas Fundamentales}

\centerline{UNED}

\centerline{acosta@mat.uned.es}

{\ }

\centerline {\bf Sabir M. Gusein-Zade $^2$,}

\centerline{Moscow State University}

\centerline{Independent University of Moscow}

\centerline{sabir@mccme.ru}

{\ }

\centerline{\bf Sergey M. Natanzon {\footnote {Partially supported
by grants RFBR-10-01-00678, NSh-709.2008.1}}}

\centerline{A.N.Belozersky Institute, Moscow State University}

\centerline{Independent University of Moscow}

\centerline{Institute Theoretical and Experimental Physics}

\centerline{natanzons@mail.ru}

{\ }

\centerline {\bf Abstract}

{\ }

Klein foams are analogues of Riemann and Klein surfaces with
one-dimensional singularities. We prove that the field of
dianalytic functions on a Klein foam $\Omega$ coincides with the
field of dianalytic functions on a Klein surface $K_{\Omega}$. We
construct the moduli space of Klein foams and we prove that the
set of classes of topologically equivalent Klein foams form
an analytic space homeomorphic to $\mathbb{R}^{n}/\mathsf{Mod}$, where $\mathsf{Mod}$ is a discrete group.

\tableofcontents

{\ }

\section{Introduction}

\label{r1}

Foams are surfaces with one-dimensional singularities. Topological foams are
exploited
in different fields of mathematical physics \cite{B,O,R,KR2} and topology \cite{K,KR1,MSV,TT,PT}. A topological foam is constructed from finitely many ordinary
surfaces with boundaries ("patches") by gluing them along segments of their
boundaries. The glued boundaries of surfaces form a "seamed graph", which is
the singular part of the complex.

Here we consider Klein foams that are analogues of Riemann and Klein surfaces
for foams. Thus we consider a foam with concordant complex structures on its
patches. "Concordant" means that there exists a dianalytic map from the foam
to the complex disk $D$. Maps of this type on a Klein foam $\Omega$ we call
dianalytic functions on $\Omega$.

In section \ref{r2} we prove that the field of dianalytic functions on a Klein
foam $\Omega$ coincides with the field of dianalytic functions on a Klein
surface $K_{\Omega}$. Moreover there exists a dianalytic map $\varphi_{\Omega
}: \Omega\rightarrow K_{\Omega}$ such that any dianalytic function on $\Omega$
is of the form $f\varphi_{\Omega}$ where $f$ is a dianalytic function on
$K_{\Omega}$.

We say that Klein foams $\Omega$ and $\Omega^{\prime}$ are topologically
equivalent, if there exist homeomorphisms $f_{\Omega}:\Omega\rightarrow
\Omega^{\prime}$ and $f_{K}:K_{\Omega}\rightarrow K_{\Omega^{\prime}}$ such
that $\varphi_{\Omega^{\prime}}f_{\Omega}=\varphi_{\Omega}f_{K}$. In section
\ref{r3} we prove that any class $M$ of topological equivalence of Klein foams
(i.e. the set of Klein foams with a fixed topological type) has a natural
analytic structure.
It is connected and homeomorphic to $\mathbb{R}^{n}/\mathsf{Mod}$,
where $\mathsf{Mod}$ is a discrete group. This gives a topological
description of the moduli space of Klein foams.

A motivation for study of the moduli space of Klein foams is string theory and
2D gravity \cite{B,R,HP,M}. In subsection \ref{r2.3}
we prove that our definition of Klein foams is compatible with cyclic foam
topological field theory \cite{N} that is a rough topological approach to the
corresponding version of the string theory.

Part of this work was done during the stays of one of the authors
(S.N.) in Max-Plank-Institute in Bonn. He is grateful to MPIM for
their hospitality and support.
\section{Foams}

\label{r2}

\subsection{Topological foams}

\label{r2.1}

We shall consider \textit{generalized graphs}, that is one dimentional spaces
consisting of (finitely many) verticies and edges, where edges are either
segments (connecting vertices) or (isolated) circles.

A \textit{topological foam} $\Omega$ is a triple $(S,\Delta,\varphi)$, where
\begin{itemize}
\item $S=S(\Omega)$ is a compact surface (2-manifold, possibly non-connected
and non-orientable) with boundary $\partial S$ (which consists of
pair-wise non-intersecting circles);
\item $\Delta=\Delta(\Omega)$ is a generalized graph;
\item $\varphi=\varphi_{\Omega}:\partial S\to\Delta$ is \textit{the gluing
map}, that is, a map such that:\newline a) $\mbox{Im\,}\varphi=\Delta
$;\newline b) on each connected component of the boundary $\partial S$,
$\varphi$ is a homeomorphism on a circle in $\Delta$;\newline c) for an edge
$l$ of $\Delta$, any connected component of $S$ contains at most one connected
component of $\varphi^{-1}(l\setminus\partial l)$.
\end{itemize}

The result of the gluing (i.e. $S\cup\Delta$ with $x$ and $\varphi(x)$
identified for $x\in\partial S$) will be denoted by $\check{\Omega}$. Let
$\Omega_{b}$ be the set of vertices of the graph $\Delta$. We say that a foam
$\Omega$ is \textit{normal} if, for any vertex $v$ of the graph $\Delta$, its
punctured neighbourhood in $\check{\Omega}$ is connected.

For a surface $S$ with boundary, the \textit{double} of $S$ is an
oriented surface $\hat{S}$ defined in the following way. Let
$\widetilde{S}$ be the \textquotedblleft
orientation\textquotedblright\ 2-fold covering over $S$. (The
surface $\widetilde{S}$ consists of pairs $(x,\delta)$, where $x$
is a points of $S$ and $\delta$ is a local orientation of $S$ at
the point $x$. If the surface $S$ is orientable, $\widetilde{S}$
is the union of 2 copies of $S$ with different orientations. The
surface $\widetilde{S}$ is oriented in the natural way. If $S$ is
a smooth surface, the surface $\widetilde{S}$ is also smooth.) The
surface $\widetilde{S}$ possesses a natural involution $\sigma$,
$\widetilde{S}/\sigma=S$. The double $\hat{S}$ of the surface $S$
is the surface $\widetilde{S}$ with the points $(x,\delta)$ and
$(x,-\delta)$ identified for $x\in\partial{S}$ ($-\delta$ is the
local orientation opposite to the orientation $\delta$).

A \textit{morphism} $f$ of topological foams $\Omega^{\prime}\rightarrow
\Omega^{\prime\prime}$ ($\Omega^{\prime}=(S^{\prime}, \Delta^{\prime}, \varphi^{\prime})$,
$\Omega^{\prime\prime}=(S^{\prime\prime},\Delta
^{\prime\prime},\varphi^{\prime\prime})$) is a pair
$(f_{S},f_{\Delta})$ of (continuous) maps $f_{S}:
\hat{S}^{\prime}\to\hat{S}^{\prime \prime}$ and $f_{\Delta}:
\Delta^{\prime}\to\Delta^{\prime\prime}$ such that $f_{S}$ is an
orientation preserving ramified covering commuting with the
natural involutions on $\hat{S}^{\prime}$ and
$\hat{S}^{\prime\prime}$, $\varphi^{\prime\prime}\circ f_{S} =
f_{\Delta}\circ \varphi^{\prime}$ and
$f_{\Delta}|_{\Delta^{\prime}\setminus\Omega^{\prime }_{b}}$ is a
local homeomorphism $\Delta^{\prime}\setminus\Omega^{\prime}_{b}$
on $\Delta^{\prime\prime}\setminus\Omega^{\prime\prime}_{b}$.


\subsection{Dianalytic foams}

\label{r2.2}

A \textit{Klein surface} \cite{AG} (see also \cite{N1}) is a
surface $S$ (possibly with boundary and/or non-orientable) with a
class of equivalence of \textit{dianalytic atlases}. A
\textit{dianalytic atlas} consists of charts
$\{(U_{\alpha},\psi_{\alpha})|\alpha\in\mathcal{A}\}$, where
$\Omega =\bigcup_{\alpha}U_{\alpha}$, $\psi_{\alpha}:
U_{\alpha}\rightarrow D\subset\mathbb{C}$ is a homeomorphism on
$\psi_{\alpha}(U_{\alpha})$ and $\psi_{\alpha}\psi_{\beta}^{-1}$
is a holomorphic or an anti-holomorphic map on
$\psi_{\alpha}(U_{\alpha})\cap\psi_{\beta}(U_{\beta})$ for any
$\alpha,\beta\in\mathcal{A}$. (Two dianalytic atlases are called
equivalent if their union is also a dianalytic atlas.)

One can see that a Klein surface is a surface $S$ with a complex analytic
structure on the double $\hat{S}$ such that the natural involution on $\hat
{S}$ is anti-holomorphic. A \textit{morphism} of Klein surfaces is an analytic
map between their doubles which commutes with the natural involutions. Moduli
space of Klein surfaces were studied in \cite{N2,N4,S}. The category of
compact Klein surfaces is isomorphic to the category of real algebraic curves
\cite{AG}.

A normal topological foam $\Omega=(S,\Delta,\varphi)$, where $S$ is a Klein
surface, is called \textit{dianalytic}. A \textit{morphism} $f$ of dianalytic
foams $\Omega^{\prime}\rightarrow\Omega^{\prime\prime}$ ($\Omega^{\prime
}=(S^{\prime},\Delta^{\prime},\varphi^{\prime})$, $\Omega^{\prime\prime
}=(S^{\prime\prime},\Delta^{\prime\prime},\varphi^{\prime\prime})$) is a
morphism $(f_{S},f_{\Delta})$ of the corresponding topological foams such that
$f_{S}$ is a morphism of Klein surfaces.

Let $\Omega_{0}$ be the dianalytic foam $(D,\partial D,I_{D})$
where $D$ is the unit disk in the complex line, and $I_{D}$ is the
tautological map. A \textit{dianalytic function} on a dianalytic
foam $\Omega$ is a morphism of $\Omega$ to $\Omega_{0}$.

Define now \textit{a Klein foam} as a dianalytic foam $\Omega=(S,\Delta
,\varphi)$ admitting an everywhere locally non-constant dianalytic function
$f_{0}$. Any Klein surface $K$ in a natural way can be considered as a Klein foam.

\begin{theorem}
\label{t2.1} For any Klein foam $\Omega$ there exists a Klein
surface $K=K_{\Omega}$ and a dianalytic morphism $\phi_{\Omega}:
\Omega\rightarrow K$ such that the correspondence $f\mapsto
f\phi_{\Omega}$ is an isomorphism between the sets (fields) of
dianalytic functions on $K$ and on $\Omega$ respectively, i.e. any
dianalytic function $f: \Omega\rightarrow\Omega_{0}$ is the
composition $f_{(K)}\phi_{\Omega}$, where $f_{(K)}:
K\rightarrow\Omega _{0}$ is a dianalytic function and vice versa,
for any dianalytic function $f_{(K)}$ on $K$, the composition
$f_{(K)}\phi_{\Omega}$ is a dianalytic function on the foam
$\Omega$.
\end{theorem}

\begin{proof}
Let us call points $q_{1}$ and $q_{2}$ of the surface $\hat{S}$
pre-equivalent if there exist paths $\gamma_{j}(t)$ on $\hat{S}$,
$j=1, 2 $, $t\in[0,1]$, such that:

\begin{itemize}
\item $\varphi(\gamma_{1}(0))=\varphi(\gamma_{2}(0))$ is an inner point of an
edge of the graph $\Delta$ (and therefore not a ramification point of the map
(dianalytic function) $f^{0}$);

\item $f^{0}(\gamma_{1}(t))=f^{0}(\gamma_{2}(t))$ for $t\in[0,1]$;

\item $\gamma_{j}(1)=q_{j}$ for $j=1,2$;

\item $f^{0}(\gamma_{j}(t))$ is not a ramification point of the map $f^{0}$
for any $t\in[0,1)$, $j=1,2$.
\end{itemize}

The values of a dianalytic function $f$ on the foam $\Omega$ at
pre-equivalent points coincide. This follows from the uniqueness
of the analytic continuation taking into account the fact that
$f^{0}$ can be considered as a local coordinate for all points
$\gamma_{j}(t)$ with $t\in\lbrack0,1)$ and the restrictions of $f$
to neighborhoods of the points $\gamma_{j}(0)$ in $\partial
S_{\Omega}$, $j=1,2$, coincide as functions of $f_{0}$.

Let us call points $q^{\prime}$ and $q^{\prime\prime}$ from
$\hat{S}$ equivalent if there exists a sequence of points
$q_{j}\in \hat{S}$, $j=0,1,\ldots,n$, such that
$q_{0}=q^{\prime}$, $q_{n}=q^{\prime\prime}$, and the point
$q_{j-1}$ is pre-equivalent to the point $q_{j}$ for
$j=1,2,\ldots,n$. Outside of (the preimage of) the set of
ramification points the described equivalence relation identifies
some points in the preimages in $\hat{S}$ (with respect to the
function $f_{0}$) of points of $\hat{D}$. If two non-ramification
points are identified, they have neighborhoods which are
identified with each other by this equivalence relation and the
identification is a homeomorphism. This means that the factor
${K}_{\Omega}$ by this equivalence relation inherits the structure
of a (closed) complex analytic curve with an anti-holomorphic
involution and with a map to $\widetilde{D}$. The factorization
map $\phi_{\Omega}$ is a (usual) covering outside of the set of
ramification points of the map $f^{0}$. A dianalytic function on
the foam $\Omega$ induces a dianalytic function on $K_{\Omega}$.
This implies that it is induced from a dianalytic function on
$K_{\Omega}$.

For a dianalytic function $f_{(K)}$ on $K_{\Omega}$, the
composition $f_{(K)}\phi_{\Omega}$ may fail to be a dianalytic
function on the foam $\Omega$ only if there exist two points
$q_{1}$ and $q_{2}$ of $\partial S_{\Omega}$ such that they map to
the same vertex of the graph $\Delta$ but are not equivalent. In
particular this means that no point of $\partial S_{\Omega}$ of a
neighborhood of the point $q_{1}$ is equivalent to one of a
neighborhood of the point $q_{2}$. This contradicts the
requirement of normality of the foam $\Omega$. This proves the
statement.
\end{proof}

We call $\phi_{\Omega}: \Omega\rightarrow\Omega_{K}$ \textit{the
canonical morphism}. An isomorphism
$f:\Omega^{\prime}\rightarrow\Omega^{\prime\prime}$ induces an
equivalence of the canonical morphisms $\phi_{\Omega^{\prime}}:
\Omega^{\prime}\rightarrow\Omega^{\prime}_{K}$ and
$\phi_{\Omega^{\prime\prime}}:
\Omega^{\prime\prime}\rightarrow\Omega^{\prime\prime}_{K}$.

\begin{corollary}
\label{c2.2} Let
$f:\Omega^{\prime}\rightarrow\Omega^{\prime\prime}$ be an
isomorphism of Klein foams. There exists an isomorphism of Klein
surfaces
$f_{K}:\Omega_{K}^{\prime}\rightarrow\Omega_{K}^{\prime\prime}$
such that
$f_{K}\phi_{\Omega^{\prime}}=\phi_{\Omega^{\prime\prime}}f$.
\end{corollary}

We say that Klein foams $\Omega^{\prime}=(S^{\prime}, \Delta^{\prime}, \varphi^{\prime})$ and $\Omega^{\prime\prime}=(S^{\prime\prime}, \Delta^{\prime\prime}, \varphi^{\prime\prime})$ have \textit{the same
topological type} if there exist isomorphisms of topological foams
$f:\Omega^{\prime}\rightarrow\Omega^{\prime\prime}$ and $f_{K}:\Omega
_{K}^{\prime}\rightarrow\Omega_{K}^{\prime\prime}$ such that $f_{K}\phi_{\Omega^{\prime}}=\phi_{\Omega^{\prime\prime}}f$.


\subsection{Strongly oriented foams}

\label{r2.3}

Klein Topological Field Theories describe rough topological approach
to the corresponding versions of string theory and Hurwitz numbers \cite{AN,AN1,AN2,AN3,CNP,LN}. It follows from
\cite{N}, that one can extend Klein Topological Field Theory to
\textit{oriented foams}. An oriented foam is a topological foam $(S,\Delta
,\varphi)$ with a special coloring of $S$.
A special coloring of $S$ exists, if and only if

\begin{itemize}
\item vertices of any connected component of $\Delta$ allow a cyclic order
that agrees with an orientation of $\partial S$ by $\varphi$;
\item $\varphi$ maps different connected components of the boundary of any
connected component of $S$ to different connected components of $\Delta$.
\end{itemize}

A special coloring defines (via $\varphi$) orientations of all edges of $\Delta$ compatible with the order of vertices.
An oriented foam $(S,\Delta,\varphi)$ is \textit{strongly oriented} if the
orientation of $\partial S$ is induced by an orientation of $S$.

\begin{lemma}
\label{l2.1} Let $\Omega=(S, \Delta, \varphi)$ be a strongly oriented foam.
Then there exists a morphism $(f_{S}, f_{\Delta})$ to the disk $($i.e. to the
topological foam $(D,\partial D, I_{D})$$)$ which is a homeomorphism on any
connected component of $\partial S$.
\end{lemma}

\begin{proof}
Since $S$ is oriented, its double $\hat S$ consists of $S$ and of its copy with the other orientation (glued along the boundary $\partial S$). The double $\hat D$ of the disk $D$ is the 2-sphere $S^2$ (with the equator $\partial D$). To construct the map $f_S: \hat S \to S^2$, one can construct a map $f:S\to S^2$ sending the boundary $\partial S$ to $\partial D$ which is a local homeomorphism in a neighborhood of $\partial S$ and a ramified covering inside $S$. The map $f$ with the canonical involutions on $\hat S$ and on $\hat D$ define the map $f_S$. Let us fix orientation preserving maps $f^{\prime}: \partial S \to \partial D$ and $f_{\Delta}:\Delta\to \partial D$ (such that $f^{\prime}=\varphi f_{\Delta}$) compatible with the orientations of $\partial S$ and of the edges of $\Delta$. Gluing a disk to each component of the boundary $\partial S$, one obtains an oriented closed surface $S^{\prime}$. Now existence of the map $f$ with the described properties follows from the fact that there exists a map from the surface $S^{\prime}$ to the 2-sphere $S^2$ which is a ramified covering of high order: greater or equal to the number $d$ of components of the boundary $\partial S$. Moreover, one can choose this map so that disk neighborhoods of $d$ fixed points are sent to a fixed disk in $S^2$ by prescribed homeomorphisms. To obtain the desired map, one should use the map of the added disks to the copy of $D$ with the inverse orientation which is the radial extension of the map $f^{\prime}$.
\end{proof}

\begin{theorem}
\label{t2.2} Any strongly oriented foam allows a dianalytic structure, turning
such foam into a Klein foam.
\end{theorem}

\begin{proof}
Consider a strongly oriented foam $\Omega=(S,\Delta,\varphi)$. From lemma
\ref{l2.1}, it follows that there exists a topological morphism $(f_{S},f_{\Delta})$ of $\Omega$ to the dianalytic disk $\Omega_{0}=(D,\partial
D,I_{D})$. Therefore there exists a (unique) dianalytic structure on $S$ such
that $f_{S}$ is a dianalytic map \cite{Ker}.
\end{proof}


\section{Moduli spaces}

\label{r3}

\subsection{Moduli of Klein surfaces}

\label{r3.1} In this and the next subsections we assume all surfaces to be
connected and to have finitely generated fundamental groups. The fundamental
group $\check{\gamma}_{+}$ of an orientable surface $K_{+}$ has a co-presentation

{\ }

$\left\langle
\begin{array}
[c]{c}
a_{1},b_{1},...,a_{g_{+}},b_{g_{+}},h_{1},...,h_{m},x_{1},...,x_{r}, e_{1},...,e_{k},c_{i1},...,c_{ib_{i}},i=1,...,k:\\
c_{ij}^{2}=1,
{\textstyle\prod}
[a_{j},b_{j}]
{\textstyle\prod}
x_{j}h_{j}
{\textstyle\prod}
e_{j}=1
\end{array}
\right\rangle $.

{\ }

\noindent The fundamental group $\check{\gamma}_{-}$ of a non-orientable
surface $K_{-}$ has a co-presentation

{\ }

$\left\langle
\begin{array}
[c]{c}
d_{1},...,d_{g_{-}},h_{1},...,h_{m},x_{1},...,x_{r},e_{1},...,e_{k}
,c_{i1},...,c_{ib_{i}},i=1,...,k:\\
c_{ij}^{2}=1,
{\textstyle\prod}
d_{j}^{2}
{\textstyle\prod}
h_{j}
{\textstyle\prod}
x_{j}
{\textstyle\prod}
e_{j}=1
\end{array}
\right\rangle $.

{\ }

The collection $\check{t}=({\varepsilon\in\{+,-\},g_{\varepsilon},m,r,k,b_{1},...,b_{k}})$ is \textit{the topological type} of the surface $K$.
Here $g$ is the (geometric) genus of $K$; $m$ is the number of holes; $r$ is
the number of interior punctures; $k$ is the number of boundary components
and $b_{i}, i=1,2, \ldots, k$, is the number of punctures on the boundary
component with the number $i$.

Any Klein surface $\check{K}$ is isomorphic to $\check{K}^{\psi}=\Lambda
/\psi(\check{\gamma})$, where $\check{\gamma}\in\{\check{\gamma}_{+},\check{\gamma}_{-}\}$, $U=D\setminus\partial D$ and $\psi:\check{\gamma
}\rightarrow\mathop{\sf Aut}\nolimits(U)$ is a monomorphism to the group
$Aut(U)$ of dianalytic automorphisms of $U$ \cite{BEGG}. Moreover the
monomorphism $\psi$ satisfies the following conditions:
\begin{itemize}
\item $\psi(c_{ij})$ and $\psi(d_{j})$ are antiholomorphic and images of all
other generators are holomorphic;
\item the automorphisms $\psi(x_{i})$ are parabolic and the automorphisms
$\psi(a_{i})$, $\psi(b_{i})$, $\psi(c_{i})$, $\psi(h_{i})$, $\psi(d^{2}_{i})$
are hyperbolic;
\item invariant curves of the hyperbolic and parabolic automorphisms satisfy
some geometric properties.
\end{itemize}

These conditions imply that the group $\psi(\check{\gamma})$ is discrete.

Monomorphisms $\psi$ satisfying these conditions will be called
\textit{admissible}.
An admissible monomorphism $\psi$ always generates a discrete subgroup of
$\mathop{\sf Aut}\nolimits(U)$ and thus a Klein surface $K^{\psi}$. The Klein
surfaces $K^{\psi^{\prime}}$ and $K^{\psi^{\prime\prime}}$ are isomorphic if
and only if the groups $\psi^{\prime}(\check{\gamma})$ and $\psi^{\prime
\prime}(\check{\gamma})$ are conjugate in the group
$\mathop{\sf Aut}\nolimits(U)$, i.e. $\psi^{\prime}(\check{\gamma})=A\psi^{\prime\prime}(\check{\gamma})A^{-1}$
for an automorphism $A\in\mathop{\sf Aut}\nolimits(U)$.

The group of homotopy classes of autohomeomorphisms of $K^{\psi}$ is naturally
isomorphic to the group $\mathsf{Mod}_{\check{t}}=\widehat{\mathsf{Mod}}(\check{\gamma})/\mathop{\sf Aut}\nolimits_{0}(\check{\gamma})$, where
$\widehat{\mathsf{Mod}}(\check{\gamma})$ is a subgroup of
$\mathop{\sf Aut}\nolimits(\check{\gamma})$ and $\mathop{\sf Aut}\nolimits_{0} (\check{\gamma})\subset\widehat{\mathsf{Mod}}(\check{\gamma})$ is the subgroup
of interior automorphisms of $\check{\gamma}$.

Therefore the moduli space $M_{\check{t}}$ of Klein surfaces of topological
type $\check{t}$ is $T_{\check{t}}/\mathsf{Mod}_{\check{t}}$, where
$T_{\check{t}}$ is the set of conjugacy classes (with respect to
$\mathop{\sf
Aut}\nolimits(U)$) of admissible monomorphisms and $\mathsf{Mod}_{\check{t}}$
acts discretely. A description of geometric properties of admissible
monomorphisms in term of coordinates on $\mathop{\sf Aut}\nolimits(U)$ gives a
homeomorphism $T_{\check{t}}\leftrightarrow\mathbb{R}^{6g_{\varepsilon
}+3m+3k+2r+b_{1}+...+b_{k}-6}$.

For compact Klein surfaces $(r=b_{1}=...=b_{k}=0)$ and for oriented surfaces
without boundary $(\varepsilon=+,k=0)$ all these facts are proved in
\cite{N2,N4,N3}. A proof for an arbitrary topological type is obtained by
simple modifications of \cite{N2,N4,N3}. Thus we have

\begin{theorem}
\label{t3.1} The space of Klein surfaces of any given topological type is
connected and homeomorphic to $\mathbb{R}^{n}/\mathsf{Mod}$ where
$\mathsf{Mod} $ is a discrete group.
\end{theorem}


\subsection{Spaces of morphisms of Klein surfaces}

\label{r3.2} Two morphisms between Klein surfaces
$f^{\prime}:S^{\prime }\rightarrow K^{\prime}$ and
$f^{\prime\prime}:S^{\prime\prime}\rightarrow K^{\prime\prime}$
are called \textit{isomorphic} (respectively \textit{topologically
equivalent}), if there exist isomorphisms (respectively
homeomorphisms) $f_{S}:S^{\prime}\rightarrow S^{\prime\prime}$ and
$f_{K}:K^{\prime}\rightarrow K^{\prime\prime}$ such that
$f_{S}f^{\prime }=f^{\prime\prime}f_{K}$. The space of isomorphic
classes of morphisms (with the natural topology) is called
\textit{the moduli space of morphisms}. A
class of topological equivalence of morphisms is called the
\textit{topological type of a morphism} of the class.

Now let us describe a class of morphisms of degree $d$ with target
a compact Klein surface of type $t=({\varepsilon,g,0,k,0,...,0})$.
Let $\check{\gamma}$ be the fundamental group of a Klein surface
of type $\check{t}=({\varepsilon ,g,r,k,b_{1},...,b_{k}})$, and
let $\tilde{\gamma}\subset\check{\gamma}$ be a subgroup of index
$d$. Consider $\psi\in T_{\check{t}}$ and $\check{K}^{\psi
}=U/\psi(\check{\gamma})$,
$\tilde{S}_{\tilde{\gamma}}^{\psi}=U/\psi (\tilde{\gamma})$. The
natural embedding $\psi(\tilde{\gamma})\subset
\psi(\check{\gamma})$ induces a morphism
$\check{f}_{\tilde{\gamma}}^{\psi
}:\check{S}_{\tilde{\gamma}}^{\psi}\rightarrow\check{K}^{\psi}$ of
degree $d$. After patching on $\check{S}_{\tilde{\gamma}}^{\psi}$
and $\check{K}^{\psi}$ the punctures generated by parabolic
shifts, we obtain a morphism
$f_{\tilde{\gamma}}^{\psi}:S_{\tilde{\gamma}}^{\psi}\rightarrow
K^{\psi}$ to a Klein surface of topological type $t$.

One can show that:

\begin{itemize}
\item each non-constant morphism of Klein surfaces is isomorphic to a morphism
of the form $f^{\psi}_{\tilde{\gamma}}$;

\item the morphisms $f^{\psi}_{\tilde{\gamma}}$ and $f^{\hat{\psi}}_{\tilde{\gamma}}$ have the same topological type for any $\psi,\hat{\psi}\in
T_{\check{t}}$; moreover they are isomorphic if and only if $A\psi
(z)A^{-1}=\hat{\psi}(z)$ for an automorphism $A\in
\mathop{\sf Aut}\nolimits(U)$ and any $z\in\check{\gamma}$;

\item the morphism $f^{\psi}_{\tilde{\gamma}}$ and $f^{\psi}_{\hat{\gamma}}$
have the same topological type if and only if $\hat{\gamma}=\alpha
(\tilde{\gamma})$, where $\alpha\in\widehat{\mathsf{Mod}}_{\check{t}}$.
\end{itemize}

For orientable Riemann surfaces $(\varepsilon=+, k=0)$ all these
facts are proved in \cite[section 1.6]{N3}. A proof for an
arbitrary topological type is obtained by simple modifications of
\cite{N3}. These arguments and Theorem \ref{t3.1} give

\begin{theorem}
\label{t3.2} The space of morphisms of Klein surfaces of any given
topological type is connected and homeomorphic to
$\mathbb{R}^{n}/\mathsf{Mod}$, where $\mathsf{Mod}$ is a discrete
group.
\end{theorem}


\subsection{Construction of Klein foams}

\label{r3.3} Now we shall apply the previous section to connected Klein foams.
We shall assume that the type $t$ of $\check{\gamma}$ is such that
$k,b_{1},...,b_{k}>0$.

Let $\tilde{\gamma}^{1},...,\tilde{\gamma}^{n}\subset\check{\gamma}$ be
subgroups of finite indices. Consider conjugate classes $C_{i,j}=\cup
_{g\in\check{\gamma}}gc_{i,j}g^{-1}$ and put $C_{i,j}^{l}=C_{ij}\cap
\tilde{\gamma}^{l}$. The group of interior automorphisms of $\tilde{\gamma
}^{l}$ splits $C_{i,j}^{l}$ into orbits $C_{i,j}^{l,1},...,C_{i,j}^{l,m_{l}}$.
A set $(C_{i,j}^{l^{1},m^{1}},...,C_{i,j}^{l^{s},m^{s}})$ is called \textit{an
edge} of $(\check{\gamma};\tilde{\gamma}^{1},...,\tilde{\gamma}^{n})$, if
$l^{v}\neq l^{w}$ for $v\neq w$. \textit{A foam system} is a collection
$\Gamma=\{\check{\gamma};\tilde{\gamma}^{1},...,\tilde{\gamma}^{n}; H\}$,
where $H$ is a set of mutually disjoint edges of $(\check{\gamma}; \tilde{\gamma}^{1},...,\tilde{\gamma}^{n})$ that contains all $C_{i,j}^{l}$.

Let $\psi\in T_{\check{t}}$. From subsection \ref{r3.2}, it follows that the
subgroups $\tilde{\gamma}^{1},...,\tilde{\gamma}^{n}$ generate a morphism
$f^{\psi}:S^{\psi}\rightarrow K^{\psi}$, where $S^{\psi}=S_{\tilde{\gamma}^{1}}^{\psi}\coprod...\coprod S_{\tilde{\gamma}^{n}}^{\psi}$. The
parabolic points of $\check{\gamma}$ divide the boundary of $S$
into segments bijectively corresponding to sets
$C_{i,j}^{l}\subset C_{i,j}$. Any element from $C_{i,j}$ is the
reflection in a straight line $l$ in the Poincar\'{e} model of
Lobachevsky geometry on $U$. The group $\check{\gamma}$ acts
transitivity on the set $\{l\}$ of these lines. The line
corresponding to $C_{i,j}^{l}\subset C_{i,j}$ forms (under the
natural projection) a segment on $\partial S^{l}$. Let
$(C_{i,j}^{l^{1},m^{1}},...,C_{i,j}^{l^{s},m^{s}})\in H$ be an
edge of $H$. Let us glue the corresponding segments of $\partial
S^{l^{1}}$, ..., $\partial S^{l^{s}}$ by the action
$\check{\gamma}$ on $\{l\}$. The continuation of the gluing to the
closures of the segments gives a graph $\Delta$ and a map
$\varphi:\partial S\rightarrow\Delta$. Let
$\Omega_{\Gamma}^{\psi}=(S,\Delta,\varphi)$.

\begin{lemma}
The triple $\Omega^{\psi}_{\Gamma}$ is a Klein foam.
\end{lemma}

\begin{proof}
First let us prove that $(S,\Delta,\varphi)$ is a topological foam. The first
two conditions and the point a) of the last condition follow directly from the
construction. The points b) and c) follow from the definition of an edge of
$(\check{\gamma};\tilde{\gamma}^{1},...,\tilde{\gamma}^{n})$ (the requirement
that $l^{v}\neq l^{w}$ for $v\neq w$) and the fact that the restriction of the
natural projection to any line from $\{l\}$ is a homeomorphism. From
subsection \ref{r3.2}, it follows that $f^{\psi}$ induces a morphism of Klein
foams $f:S\rightarrow K^{\psi}$. Moreover, the Klein surface $K^{\psi}$ admits
a morphism to the unit disk \cite{AG}.
\end{proof}

Now we shall construct a class of foam systems corresponding to any connected
Klein foam $\Omega=(S,\Delta,\varphi)$. Let us consider the Klein surface
$K=K_{\Omega}\in M_{t}$ and the canonical morphism $\phi_{\Omega}:
\Omega\rightarrow K$ from theorem \ref{t2.1}. Let $B\subset K$ be the set of
critical values of $\phi=\phi_{\Omega}$ and $\check{K}=K\setminus B\in
M_{\check t}$. Let $\check{S}:=\phi^{-1}(\check{K})$.

From subsection \ref{r3.1}, it follows that $\check{K}=\Lambda/\psi
(\check{\gamma})$ for $\psi\in T_{\check{t}}$, where $\check{t}$ is the type
of $\check{\gamma}$. From subsection \ref{r3.2}, it follows that the
restrictions of $\phi$ to connected components of $\check{S}$ are isomorphic
to morphisms generated by a family of subgroups $\tilde{\gamma}^{1}, \ldots, \tilde{\gamma}^{n}\subset\check{\gamma}$. The edges of $\Delta$ generate
the family of edges $H$ of $(\check{\gamma},\tilde{\gamma}^{1},...,\tilde
{\gamma}^{n})$ and thus a foam system $\Gamma=\{\check{\gamma};\tilde{\gamma
}^{1},...,\tilde{\gamma}^{n}; H\}$. From subsection \ref{r3.3}, it follows
that the Klein foams $\Omega^{\psi}_{[\Gamma]}$ and $\Omega$ are isomorphic.
Thus we have proved

\begin{lemma}
For any Klein foam $\Omega$, there exists a foam system $\Gamma$ and $\psi\in
T_{\check{t}}$ such that $\Omega$ and $\Omega^{\psi}_{\Gamma}$ are isomorphic.
\end{lemma}


\subsection{Moduli of Klein foams}

\label{r3.4} Let $\Omega^{\prime}$ and $\Omega^{\prime\prime}$ be Klein foams
with canonical morphisms $\phi_{\Omega^{\prime}}: \Omega^{\prime} \rightarrow\Omega^{\prime}_{K}$ and $\phi_{\Omega^{\prime\prime}}:
\Omega^{\prime\prime}\rightarrow\Omega_{K}^{\prime\prime}$. We say that
$\Omega^{\prime}$ and $\Omega^{\prime\prime}$ have \textit{the same
topological type}, if there exists an isomorphism $(f_{S}, f_{\Delta})$ of the
topological foams, $f_{S}: \widetilde{S}^{\prime}\to\widetilde{S}^{\prime\prime}$, $f_{\Delta}:\Delta^{\prime}\to\Delta^{\prime\prime}$, and a
homeomorphism $f_{K}: \widetilde{K}^{\prime}\to\widetilde{K}^{\prime\prime} $
such that $\phi_{\Omega^{\prime\prime}}f_{S}=f_{K}\phi_{\Omega^{\prime}}$.

The space of isomorphic classes of Klein foams (with the natural topology) is
called the \textit{moduli space of Klein foams}. A
class of topological equivalence of morphisms is called a
\textit{topological type of a Klein foam}.

An element $(\alpha^{1},...,\alpha^{n})\in\widehat{\mathsf{Mod}}(\check
{\gamma}) ^{\otimes n}$ acts on the set $\{\Gamma\}$ of foam systems sending
$\tilde{\gamma}^{l}$ to $\alpha^{l}(\tilde{\gamma}^{l})$ and $C_{i,j}^{l}$ to
$\alpha^{i}(C_{i,j}^{l})$ $(l=1,...,l)$. The orbit $[\Gamma]$ of a foam system
$\Gamma$ under the action of $\widehat{\mathsf{Mod}}(\check{\gamma})^{\otimes
n}$ is called the \textit{class of $\Gamma$}. The subsections \ref{r3.1} and
\ref{r3.2} imply the following statement.

\begin{lemma}
The Klein foams $\Omega^{\psi}_{\Gamma}$ and
$\Omega^{\hat{\psi}}_{\Gamma}$ have the same topological type for
any $\psi,\hat{\psi}\in T_{\check{t}}$. Moreover they are
isomorphic if $A\psi(z)A^{-1}=\hat{\psi}(z)$ for an automorphism
$A\in\mathop{\sf Aut}\nolimits(U)$ and any $z\in\check{\gamma}$.
The morphisms $f^{\psi}_{\tilde{\gamma}}$ and
$f^{\psi}_{\hat{\gamma}}$ have
the same topological type if and only if $\hat{\gamma}=\alpha(\tilde{\gamma})$, where $\alpha\in\widehat{\mathsf{Mod}}_{\check{t}}$.
\end{lemma}

Thus the topological type of $\Omega^{\psi}_{\Gamma}$ is defined by $[\Gamma]
$. Let $\mathsf{Mod}_{[\Gamma]}:=\widehat{\mathsf{Mod}}(\check{\gamma
})^{\otimes n}/\mathop{\sf Aut}\nolimits_{0}(\check{\gamma})$. The set of all
Klein foams of topological type $[\Gamma]$ is in one-to-one correspondence
with $T_{t}/\mathsf{Mod}_{[\Gamma]}$. Therefore

\begin{theorem}
\label{t3.3} The space of Klein foams of a given topological type
is connected and homeomorphic to $\mathbb{R}^{n}/\mathsf{Mod}$,
where $\mathsf{Mod} $ is a discrete group.
\end{theorem}

\end{document}